%% file: Lefschetz_embeddings_and_its_obstructions.tex
\documentclass[11pt]{amsart}
\usepackage{amssymb}

\usepackage{color}
\usepackage{graphics}
\usepackage{amsmath,amscd}
\usepackage{amsthm}
\usepackage[all,cmtip]{xy}
\usepackage{tikz}
\usepackage{graphicx}
\usepackage{epstopdf}
\usepackage{mathtools}
\usepackage{amssymb}
\usepackage{amsbsy}
\usepackage{bm}
\usepackage{tikz-cd}
\usepackage{enumerate}
\usepackage{enumitem}
\usepackage{float}
\theoremstyle{definition}

\usepackage[colorlinks=green,linkcolor=blue,citecolor=red,urlcolor=red]{hyperref}

\newcommand{\C}{\mathbb{C}}

\newtheorem{proposition}{Proposition}[section]
\newtheorem{theorem}[proposition]{Theorem}
\newtheorem{definition}[proposition]{Definition}
\newtheorem{lemma}[proposition]{Lemma}

\newtheorem{corollary}[proposition]{Corollary}
\newtheorem{remark}[proposition]{Remark}

\begin{document}
	
	\title{A note on embedding of achiral Lefschetz fibrations}
	
	\subjclass{Primary: 57R40. Secondary: 57M50, 57R17.}
	
	\keywords{Achiral Lefschetz fibration, Open book, Embedding.}

	\author{Arijit Nath}
	\address{Indian Institute of Technology Madras}
	\email{arijit2357@gmail.com}
	
	\author{Kuldeep Saha}
	\address{Indian Institute of Science Education and Research Bhopal}
	\email{kuldeep.saha@gmail.com}
	
	\begin{abstract}
		We discuss $4$-dimensional achiral Lefschetz fibrations bounding $3$-dimensional open books and study their Lefschetz fibration (LF) embedding in a bounded $6$-dimensional manifold, in the sense of Ghanwat--Pancholi. As an application we give another proof of the fact that every closed orientable $4$-manifold embeds in $S^2 \times S^2 \times S^2$. We also prove that every achiral Lefschetz fibration with hyperelliptic monodromy admits LF embedding in $D^6 = D^2 \times D^4$ and discuss an obstruction to such LF embeddings.
	\end{abstract}
	
	\maketitle

	\section{Introduction}
	
	In recent times, various embedding problems for manifolds with extra geometric structures have been investigated. In particular, the study of embeddings in the category of open books and Lefschetz fibrations has seen some progress. The open book case was studied in \cite{EL},\cite{PPS} and \cite{S}. Recently, Ghanwat and Pancholi have proved the existence of embedding of any oriented closed $4$-manifold in both $S^2 \times S^2 \times S^2$ and $\mathbb{C}P^3$. In the present article, we prove a relative version of the main Theorem (see Theorem \ref{GP main theorem}) in \cite{GP} to show the following.
	
		\begin{theorem}\label{main theorem} Let $V^4 = LF(\Sigma,\phi).$	$V^4$ admits a relative LF embedding in $(S^2 \times S^2 \setminus D^4) \times D^2.$
		\end{theorem}
	
\noindent Here, $LF(\Sigma,\phi)$ describes an achiral Lefschetz fibration with fiber $\Sigma$ and $\phi$ is a representation of an element in $\mathcal{MCG}(\Sigma,\partial \Sigma)$. For exact definitions we refer to section \ref{LF def}.
	
\	

Using Theorem \ref{main theorem}, we give an alternate proof of the following fact proved in \cite{GP}. 

\begin{corollary}\label{corollary}
	Every closed orientable $4$-manifold embeds in $S^2 \times S^2 \times S^2$.
\end{corollary}

\noindent Note that Corollary \ref{corollary} implies that every closed orientable $4$-manifold embeds in $\mathbb{R}^7$. However, this can be proved in a more straightforward way by showing relative LF embedding in $D^5 \times D^2.$ In a recent work by Ghanwat-Pancholi-Saha \cite{GNS}, it was shown that every $4$--manifold (possibly non-orientable) without handles of index $3$ and $4$ admits a relative $LF$ embedding in $D^5 \times D^2.$

\

The main idea of proof of Theorem \ref{main theorem} is essentially the same as that of Theorem \ref{GP main theorem} in Ghanwat--Pancholi \cite{GP}. Theorem \ref{main theorem} should be thought of as a relative version of Theorem \ref{GP main theorem}.   

\

We also discuss relative LF embeddings in $D^6$. Let $\Sigma_{g,1}$ denote a genus-$g$ surface with one boundary component. Then $\mathcal{MCG}(\Sigma_{g,1},\partial \Sigma_{g,1})$ is generated by Dehn twists along the set of curves, $\{a_1,c_1,a_2,c_2,...a_{g-1},c_{g-1},a_g,b_1,b_2\}$, as shown in Figure \ref{intropic}. These curves are known as the \emph{Humphreys generators}. The group generated by all the Humphreys generators, except $b_2$, is called the \emph{hyperelliptic subgroup} of $\mathcal{MCG}(\Sigma_{g,1}, \partial \Sigma_{g,1})$. Any element of this group is called hyperelliptic.

\begin{figure}[htbp] 
	\centering
	\includegraphics[height=3cm,width=12cm]{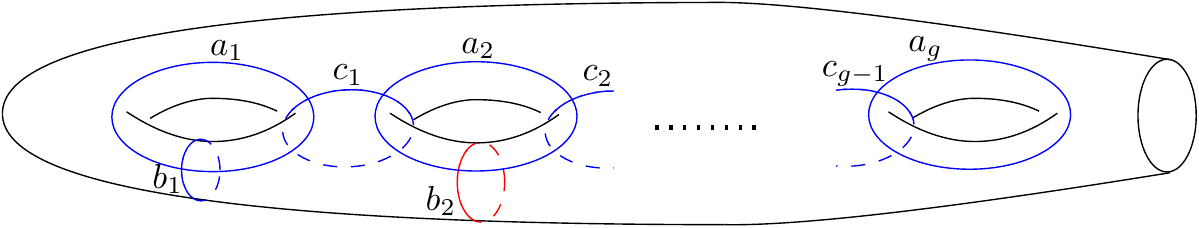}
	\caption{Humphreys generators of mapping class groups of $\Sigma_{g,1}$}
	\label{intropic}
\end{figure}

	\begin{theorem} \label{theorem 2} Let $V^4 = LF(\Sigma_{g,1},\phi)$.

		\begin{enumerate}
			
			\item If $V^4$ admits a proper embedding in $D^6$, then $V^4$ is spin.
			
			\item If $\phi$ is hyperelliptic, then $V^4$ admits a relative LF embedding in $D^6.$
			
		\end{enumerate}
	\end{theorem}

We note that Theorem \ref{theorem 2} also induces an embedding of the boundary $\partial V$ in $\partial D^6 = S^5$ preserving the open book structures coming from the respective achiral Lefschetz fibrations.

\begin{remark}
	An achiral Lefschetz fibration depends on the presentation of the monodromy of its fiber. Therefore, in statement $(2)$ of Theorem \ref{theorem 2}, the monodromy is assumed to be presented in terms of the Humphrey's generators. 
\end{remark}

\subsection{Acknowledgment} The first author is supported by the CSIR, India (Fellowship Ref. no. 09/084(0688)/2016-EMR-I).
	
	\section{Preliminaries}

	\subsection{Open book decomposition} An open book is a decomposition of a manifold into a co-dimension $2$ submanifold and a fibration over $S^1$.

	\begin{definition}\label{def:ab_open_book}
		Let $\Sigma$ be a surface with boundary and let $\phi$ be a diffeomorphism of $\Sigma$ that is identity on a collar neighborhood of $\partial \Sigma$. An open book decomposition of $M$ is a pair $(\Sigma, \phi)$ such that $M$ is diffeomorphic to $\mathcal{MT}(\Sigma, \phi) \cup_{id} \partial \Sigma \times D^2$, where $id$ denotes the identity map of $\partial \Sigma \times S^1.$
		
		Here, $\mathcal{MT}(\Sigma, \phi)$ is the mapping torus of $\phi$. We denote such an open book by $\mathcal{O}b(\Sigma, \phi)$.
	\end{definition}

	 For more details on open books, we refer to \cite{Et}.

	\subsection{Lefschetz fibration and achiral Lefschetz fibration}\label{LF def} The following descriptions of a Lefschetz fibration and an achiral Lefschetz fibration is taken from \cite{EL}.
	
	 A \emph{Lefschetz fibration}(LF) of an oriented $4$-manifold $X^4$ is a map $\pi : X^4 \to D^2$, where $D^2$ is a $2$-disk and $D\pi$ is surjective at all but finitely many points $p_1,\ldots, p_k$, called \emph{singular points}. Each point $p_i$ has a neighborhood $U_i$ that is orientation preserving diffeomorphic to an open set $U_i' \subset \C^2$, and in these local coordinates $\pi$ is given by the map $(z_1,z_2) \mapsto z_1 \cdot z_2$. We say $\pi : X^4 \to D^2$ is an \emph{achiral Lefschetz fibration}(ALF) if there is a map $\pi : X^4 \to D^2$ as above, except the local charts expressing $\pi$ as $(z_1,z_2) \mapsto z_1 \cdot z_2$, need not be orientation preserving. Following are some well known facts about (achiral) Lefschetz fibrations.
	
	\begin{enumerate}
		\item Let $F = D^2 \setminus \pi(\{p_1,\ldots, p_k\})$ and $X' = \pi^{-1}(F)$.  Then $\pi|_{X'} : X'\to F$ is a fibration with fiber some surface $\Sigma$, possibly with boundary. 
		
		\item Fix a point $x\in F$ and for each $i=1,\ldots, k$, let $\gamma_i$ be a path in $D^2$ from $x$ to $\pi(p_i)$ whose interior is in $F$. Then there is an embedded simple closed curve $v_i$ in $F_x=\pi^{-1}(x)$ that is homologically non-trivial in $F_x$ but is trivial in the homology of $\pi^{-1}(\gamma_i)$. The curve $v_i$ is called the \emph{vanishing cycle of $p_i$}. We will assume that $\gamma_i\cap \gamma_j=\{x\}$ for all $i\not=j$. 
		
		\item For each $i$, let $D_i$ be a disk inside $D^2$ containing $\pi(p_i)$ in its interior, disjoint from the $\gamma_j$ for $j\not=i$, and intersecting $\gamma_i$ in a single arc that is transverse to $\partial D_i$. The boundary $\partial \pi^{-1}(D_i)$ is a $\Sigma$-bundle over $S^1$. Identifying a fiber of $\pi^{-1}(\partial D_i)$ with $\Sigma$ using $\gamma_i$, the monodromy of $\pi^{-1}(\partial D_i)$ is given by a positive Dehn twist along $v_i$. 
		
		\item Let $D$ be a disk containing $x$ and intersecting each $\gamma_i$ in an arc transverse to $\partial D$. The manifold $X$ can be built from $D^2\times \Sigma=\pi^{-1}(D)$ by adding a $2$--handle for each $p_i$ along $v_i$ sitting in $F_{q_i}=\pi^{-1}(q_i)$ where $q_i=\partial D\cap \gamma_i$ with framing one less than the framing of $v_i$ given by $F_{q_i}$ in $\partial D^2\times \Sigma$. Conversely, any 4-manifold constructed from $D^2\times \Sigma$ by attaching $2$--handles in this way will correspond to a Lefschetz fibration.
		
		\item If $\pi: X\to D^2$ is an achiral Lefschetz fibration (ALF), then the above statements are still true but for an \emph{achiral singular point} the monodromy is a negative Dehn twist and the $2$--handle is added with framing one greater than the surface framing. 
		
		\item An ALF determines a factorization of the fiber monodromy. 
		
		\item If $\Sigma$ is a surface with boundary, then the ALF induces an open book decomposition on the boundary manifold with fiber $\Sigma$ and mnodromy same as the monodromy of the ALF.
	\end{enumerate}

 We shall denote an achiral Lefschetz fibration over $D^2$, with fiber $\Sigma$ and a monodromy $\phi$, by $LF(\Sigma,\phi)$.

    \subsection{A property of $\mathcal{MCG}(\Sigma,\partial \Sigma)$} Let $(\Sigma,\partial \Sigma)$ be an orientable bounded surface with $g$ handles and $m$ boundary components as shown in Figure \ref{pic0} below.

    \begin{figure}[htbp]
    	\centering
    	\def\svgwidth{7cm}
    	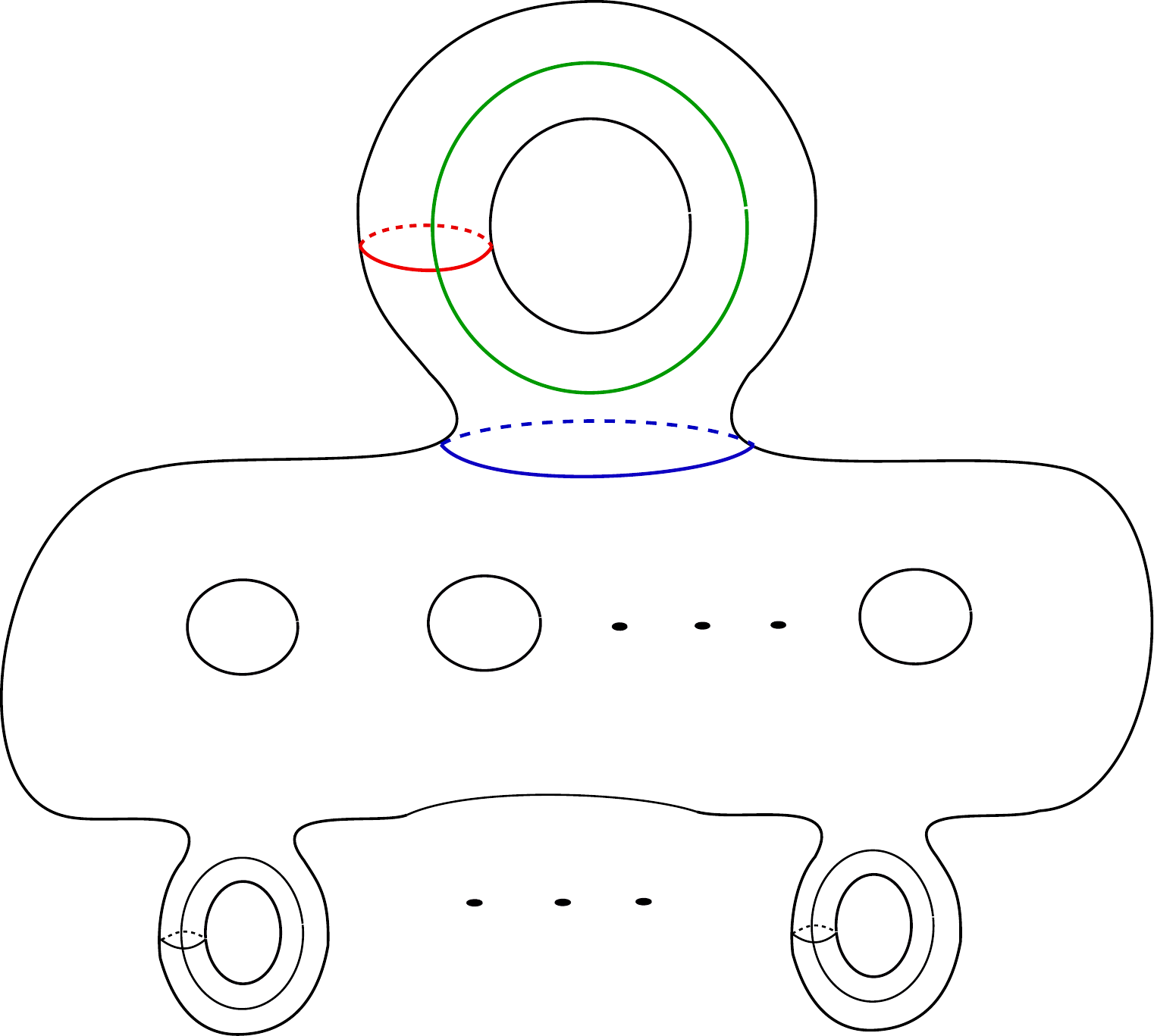
    	\caption{Genus $g$ surface.}
    	\label{pic0}
    \end{figure}

    Let $\{\alpha_1,\beta_1,\cdots,\alpha_g, \beta_g\}$ be the simple closed curves along the $g$ handles. If a simple closed curve $\gamma$ intersects $\alpha_i$ for some $i \in \{1,2,\cdots, g\}$, then we say that $\gamma$ meets the $i^{th}$ handle. Let $\tilde{\Sigma}$ be the closed surface obtained from $\Sigma$ by filling the boundary circles with disks. Let $C$ be any simple closed curve on $\tilde{\Sigma}$. Now, according to Lemma $3$ in \cite{Li}, given any such curve $C$, there exists a diffeomorphism of $\tilde{\Sigma}$ which sends $C$ to a curve which  does not meet any handle. The proof of Lemma $3$ \cite{Li} also holds for surfaces with boundary. This is because the required diffeomorphism in Lemma $3$ of \cite{Li} is obtained by applying appropriate Dehn twists along finitely many curves and therefore, the diffeomorphism is supported away from the $m$ disks bounding the $\partial_i$s in $\tilde{\Sigma}$.
    
    \begin{proposition}[Lickorish \cite{Li}] \label{lickorish} 
    	
    	Let $C$ be any simple closed curve on $(\Sigma,\partial \Sigma)$. There exists
    	a diffeomorphism $\phi \in \mathcal{MCG}(\Sigma, \partial \Sigma)$ such that $\phi(C)$ does not meet any handle of $\Sigma.$
    
    \end{proposition} 
    
    \subsection{Flexible embedding in standard position}\label{standard flex}
    
    We now recall the notion of \emph{flexible embedding} and embedding in a \emph{standard position} from \cite{GP}. We shall state the definitions adapted to the relative case which will be useful for our purpose.
    
    \begin{definition}[Flexible embedding]
    	Let $W^4$ be an orientable bounded smooth manifold and let $(\Sigma,\partial \Sigma)$ be a bounded orientable
    	surface. A smooth proper embedding $h:(\Sigma,\partial \Sigma) \hookrightarrow (W, \partial W)$ is said to be flexible if for all $\phi \in \mathcal{M}CG(\Sigma_g)$ there exists a relative diffeomorphism $\psi$ of $(W,\partial W)$ isotopic to the identity which maps $(\Sigma,\partial \Sigma)$ to itself and satisfies $\psi \circ h = h \circ \phi.$
    \end{definition}

    \begin{definition} [Embedding in standard position]
    	A proper embedding $\phi: (\Sigma,\partial \Sigma) \hookrightarrow (W,\partial W)$ is said to be in a standard position if the following properties hold:
    	
    	\begin{enumerate}
    		\item Every simple closed curve $\gamma$ on $\phi(\Sigma)$ is a boundary of a $2$--disk 
    		$ D^2$  intersecting $\phi(\Sigma)$ only in $\gamma.$

    		\item There exists a tubular neighborhood $N(D^2)$ of the disk $D^2$ having the boundary $\gamma$ such that 
    		$N(D^2)$ is the image of a coordinate chart $\phi_{\gamma}: \C^2 \rightarrow N(D^2)$ satisfying the following:
    		
    		$\phi_{\gamma}^{-1} (\phi(\Sigma) \cap N(D^2))$ is  $g^{-1}(1),$ where $g: \C^2 \rightarrow \mathbb{C}$ is the polynomial  map $g(z_1, z_2) = z_1.z_2.$ Topologically this means that $\phi_{\gamma}^{-1} (\phi(\Sigma) \cap N(D^2))$ is a Hopf annulus in a $3$-sphere around the the origin in $\mathbb{C}^2$.
    	\end{enumerate}
    	\label{def:standard position}
    \end{definition}

Next we recall the notion of a \emph{separable Hopf link}.
    
    \begin{definition}[Separable Hopf link] \label{def:separable_Hopf_link}We say that a link $l_1\sqcup l_2$ in  a $4$--manifold $W$ is a separable Hopf link provided following properties are satisfied:
    	\begin{enumerate}
    		\item There exist an embedding of a $4$--ball $D^4 = D^2\times D^2$ in $W$ such that $\partial D^2 \times \lbrace 0 \rbrace \sqcup\lbrace 0 \rbrace \times \partial D^2 = l_1\sqcup l_2$.
    		\item There exists two disjoint properly embedded discs $D_1$ and $D_2$ in $W \setminus (D^2 \times D^2)^\circ$ such that $\partial D_1 = l_1$ and $\partial D_2 = l_2$.
    	\end{enumerate}
    \end{definition}

The following Lemma was proved by Ghanwat--Pancholi (Lemma $4.4$,\cite{GP}).
    
    \begin{lemma}[\cite{GP}]\label{GP1}
    	Let $N$ be a $4$--manifold which admits a separable Hopf link. Then there exists an embedding $\phi$ of any closed orientable surface $\Sigma_g$
    	of genus $g$ in $N$ which satisfies the following:
    	
    	\begin{enumerate}
    		\item The embedding is flexible.
    		\item The embedding is in  a standard position.  
    	\end{enumerate}
    \end{lemma}

    \subsection{Lefschetz fibration embedding}

       We observe that the following relative version of Lemma \ref{GP1} also holds true. The proof is essentially the same as the proof of Lemma \ref{GP1}. However, for the sake of convenience, we review the arguments adapted to our setting, i.e., the relative case.
    
    \begin{lemma}\label{rel GP1}
    	Let $W$ be a $4$--manifold with connected boundary which admits a separable Hopf link. Then there exists a proper embedding $\phi$ of any bounded orientable surface $(\Sigma, \partial \Sigma)$ in $(W,\partial W)$ which satisfies the following:
    	
    	\begin{enumerate}
    		\item $\phi$ is flexible.
    		\item $\phi$ is in a standard position.  
    	\end{enumerate}
    \end{lemma}

    \begin{proof}[Proof of Lemma \ref{rel GP1}]
    	
    	Let $l_1\sqcup l_2$ be a separable Hopf link in $W$. So, there exists an embedded $4$--ball $D^4 = D^2\times D^2$ in $W$ such that $\partial D^2\times\lbrace 0 \rbrace \sqcup\lbrace 0 \rbrace \times \partial D^2=l_1\sqcup l_2$, and there exist two disjoint properly embedded discs $D_1$ and $D_2$ in $W\setminus (D^2\times D^2)^\circ$ such that $\partial D_1=l_1$ and $\partial D_2=l_2$. We regard a $4$--ball $D^4$ as the $4$--ball $B^4(0,2)$ of radius $2$ in $\mathbb{C}^2$ with its center at the origin. We will also regard $S^3 \times [1,2]$ as the collar $B^4(0,2) \setminus B^4(0,1)$ contained in $W$.
    	
    	Observe that the link $l_1\times\lbrace \frac{3}{2}\rbrace \sqcup l_2\times\lbrace \frac{3}{2}\rbrace$ bounds a Hopf annulus, say $\mathcal{H}$, in $S^3\times \lbrace \frac{3}{2}\rbrace $.  We embed $\Sigma$ in $S^3\times \{\frac{3}{2}\} \subset S^3 \times [1,2] \subset W$ in the following way. Let $\tilde{\Sigma}_g$ be the closed surface of genus $g$ obtained from $\Sigma$ by killing its $m$ boundary components $\{\partial_1,\cdots,\partial_m\}$, by attaching disks. The standard embedding of $\tilde{\Sigma}$ in $S^3$ is the one that bounds a genus $g$ handlebody. This induces an embedding of $\Sigma$ in $S^3$. We then remove a disk from $\Sigma$ and attach a Hopf band along its boundary. Thus, we get in $S^3\times\lbrace \frac{3}{2}\rbrace$, a surface $\widehat{\Sigma}$ with two extra boundary components as shown in Figure~\ref{pic1}. By adding the disjoint cylinders $(l_1 \sqcup l_2 \sqcup \partial_1 \sqcup \cdots \sqcup \partial_m) \times [\frac{3}{2},2]$ and two disjoint disc $\mathcal{D}_1$,$\mathcal{D}_2$ to $\widehat{\Sigma}$, we obtain a proper embedding of $\Sigma$ in $W$. Let us denote this embedding by $\phi$. We claim that the embedding $\phi: (\Sigma,\partial \Sigma) \hookrightarrow (W,\partial W)$ is both flexible and in standard position.

    	Consider a Dehn twist $\tau_\gamma$ along along a curve $\gamma$ on $\Sigma$ embedded in $W$ via $\phi$. If $\phi(\gamma)$, up to isotopy, has a Hopf annulus neighborhood in $S^3 \times \{ \frac{3}{2}\} \subset W$, then $\tau_\gamma$ can be induced by a diffeomorphism of $S^3$ that is isotopic to the identity map. As shown in the proof of Lemma $15$ in \cite{PPS}, this implies that there exists a diffeomorphism of $(W,\partial W)$, relative isotopic to the identity, which induces $\tau_\gamma$ on the embedded $\Sigma$. Note that every Lickorish generator curve has a Hopf annulus neighborhood under the embedding $\phi$. Therefore, the claim of flexibility follows by successive application of ambient relative isotopies of $(W,\partial W)$ inducing Dehn twists along Lickorish generators. See \cite{PPS} for details.
    	
    	We now show that the embedding is in  a standard position. By construction, any simple closed curve on $\phi(\Sigma)$ can be isotoped on the surface $\phi(\Sigma)$ so that it is contained in $\phi(\Sigma) \cap S^3 \times \{\frac{3}{2}\}.$  We claim that any Lickorish generator on $\phi(\Sigma)$ as well as any curve which does not meet any handle of
    	$\phi(\Sigma)$, satisfy both the properties necessary for an embedding to be in standard position. The reasons are the following.
    	
    	\begin{enumerate}
    		\item  All  curves mentioned in the claim are unknots in $S^3 \times \{\frac{3}{2}\}$. Therefore, they bound a disk in $S^3 \times [1,\frac{3}{2}],$ that meets $\phi(\Sigma)$ only in the given curve.
    		
    		\item Any curve $\gamma$ as in the claim admits a neighborhood $\mathcal{N}(C)$ in  $\phi(\Sigma)$ which is a Hopf band in $S^3 \times \{\frac{3}{2}\}.$
    	\end{enumerate}
    	
    	It follows from both the properties listed above that any  curve $C,$ which
    	is either a Lickorish generator or does not meet any handle of $\Sigma$, satisfies both the properties necessary for a surface to be in standard position.

    	By Proposition \ref{lickorish}, given any curve $C$, there exists a boundary preserving diffeomorphism of $\phi(\Sigma)$ which sends $C$ to a curve which  does not meet any handle. Since the embedding $\phi$ of $(\Sigma, \partial \Sigma)$ is flexible in $(W, \partial W)$, given any curve $C$ that meets some handles can be isotoped so that it does not meet any handle. Hence, the embedding can be assumed to be in a standard position. 
    	
    	\begin{figure}
    		\centering
    		\def\svgwidth{12cm}
    		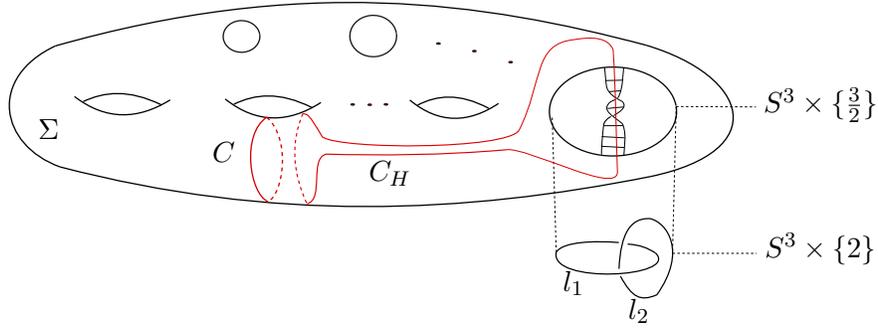
    		\caption{Flexible embedding in standard position.}
    		\label{pic1}
    	\end{figure}
    \end{proof}

    We recall the notion of a \emph{Lefschetz fibration embeddings} from \cite{GP}. The map  $\pi_2 : N \times \C P^1 \rightarrow \C P^1$ 
    corresponds to projection on the second factor.

    \begin{definition} [Lefschetz fibration embedding] \label{LF embedding}
    	Let $(M,\pi:M \rightarrow \Sigma)$  be a Lefschetz fibration, where $\Sigma$ is 
    	$2$--disk or 
    	$\mathbb{C}P^1$.  An embedding $f: M \to N\times\mathbb{C}P^1$ of a manifold $M$ into a manifold $N \times\mathbb{C}P^1$ is said to be a \emph{Lefschetz fibration embedding} provided
    	$\pi_{2}\circ f$=$  i \circ \pi,$ where $i$ is  an inclusion of $\mathbb{D}^2$
    	in $\mathbb{C} P^1$ when $\partial M \neq \emptyset,$ otherwise it is  the identity.
    \end{definition}

\begin{theorem}[Ghanwat--Pancholi, \cite{GP}] 
	\label{GP main theorem}
	Let $M$ be an orientable smooth $4$--manifold. Let $N$ be a $4$--manifold which admits a separable Hopf link. 
	If $\pi:M \rightarrow \Sigma,$
	where $\Sigma$ is either $\mathbb{C}P^1$ or a $2$--disk $\mathbb{D}^2$,
	is a Lefschetz fibration of $M$ having genus $g$ closed surfaces as fibers with $g \geq 1$, then there exists a Lefschetz fibration 
	embedding of $(M, \pi)$ in $(N\times\mathbb{C}P^1, \pi_2).$ 
\end{theorem}

We now state the definition of relative Lefschetz fibration embedding.

    \begin{definition}[Relative LF embedding] \label{RLF embedding}
    	Let $(V^4,\pi:V \rightarrow D^2)$  be an ALF. A proper embedding $f_r: (V, \partial V) \to (N \times D^2, \partial (N \times D^2))$ is said to be a \emph{relative Lefschetz fibration embedding}, if
    	$\pi_{2} \circ f_r = \pi.$
    \end{definition}

    \subsection{Spin structures on ALF with fiber closed surfaces}\label{sec spin} A smooth manifold $M$ is called \emph{spin} if its second Stiefel-Whitney class $w_2(M)$ vanishes. Let $\Sigma_g$ denote a closed oriented surface of genus $g$. Suppose that the homology classes of the vanishing cycles of $LF(\Sigma_g,\phi)$ are denoted by $v_1,\cdots, v_t \in H_1(\Sigma_g,;\mathbb{Z}_2)$. Stipsicz gave the following criteria for achiral Lefschetz fibrations with fiber $\Sigma_g$, not to be spin.
    
    \begin{theorem}[Stipsicz] \label{stipsicz}
    	A Lefschetz fibration $\pi : X \to D^2$ with fiber $\Sigma_g$, is not spin if and only if there exist vanishing cycles $v_1,\cdots, v_k$ such that their sum $v = \Sigma^k_{i=1}v_i$ also corresponds to a vanishing cycle, and $k + \Sigma_{1\leq i < j \leq k}v_i \cdot v_j \equiv 0 \pmod 2$.
    \end{theorem}

Here, $v_i \cdot v_j$ denotes the algebraic intersection number between $v_i$ and $v_j$.

	\section{Existence of codimension $2$ proper embedding of $LF(\Sigma,\phi)$} \label{sec3}
	
	To prove Theorem \ref{main theorem} we need the following Lemma.

	\begin{lemma} \label{main lemma}
		Let $\pi: V^4 \rightarrow D^2$ be a Lefschetz fibration. Let $(W^4,\partial W^4)$ be a $4$--manifold which admits a separable Hopf link. There exists a relative LF embedding of $(V, \pi)$ in $(W \times D^2, \pi_2).$ 
	\end{lemma}

	Lemma \ref{main lemma} can been seen as a relative version of Theorem \ref{GP main theorem} and the main arguments behind their proofs are essentially the same. We mostly follow the proof of Theorem \ref{GP main theorem} (Theorem $4.8$, \cite{GP}) and refer to \cite{GP} for further details.

	\begin{proof}[Proof of Lemma \ref{main lemma}]
		
		Let $c_1, c_2,  \cdots c_k$ be the $k$ critical points of the Lefschetz fibration $(V, \pi)$ and let $p_1, p_2, \cdots p_k$ be their images in $D^2$ under the map $\pi$. Let $\gamma_i$ be the vanishing cycle corresponding to the critical value $p_i$ on a generic fiber $\Sigma$ of the LF. Let $(U_i,z_1,z_2)$ be a complex co-ordinate on a disk neighborhood of $c_i$ in $V^4$ such that $\pi$ is given by $(z_1, z_2) \mapsto z_1.z_2.$ Let $\widetilde{D}_i = \pi(U_i) \subset D^2.$ and let $D_i$ be an open disk containing $p_i$ with $\overline{D_i} \subset \widetilde{D}_i$ for $i = 1,\cdots,k$. The $\widetilde{D}_is$ are disjoint.
		
		By Lemma \ref{rel GP1}, we can take a proper embedding $\iota$ of $(\Sigma,\partial \Sigma)$ in $(W^4,\partial W^4)$ which is both flexible and is in standard position. Using the flexibility of the embedding $\iota,$ we construct an embedding $\widehat{f}$ of
		$V \setminus \sqcup_{i=1}^{k} \pi^{-1}(D_i)$ in $ W \times \left( D^2 \setminus \sqcup_{i=1}^k D_i\right)$ such that the following diagram commutes:

		\begin{equation}
		\begin{tikzcd}
		V \setminus \sqcup_{i=1}^{k} \pi^{-1}(D_i) \arrow[r, "\widehat{f}"] \arrow[d,"\pi" ] &  W \times \left( D^2 \setminus \sqcup_{i=1}^k D_i\right) \arrow[d,"\pi_2"]  \\
		D^2 \setminus \sqcup_{i=1}^{k}D_i \arrow[r, "Id"] & D^2 \setminus \sqcup_{i=1}^{k}D_i.
		\end{tikzcd}
		\end{equation}

		Since the embedding $\iota$ of $(\Sigma,\partial \Sigma)$ in $(W,\partial W)$ is standard, by Lemma~\ref{rel GP1}, there exists an embedding $\Psi$ of the mapping torus $\mathcal{M}T (\Sigma, \psi)$ in $M  \times S^1$, for all $\psi \in \mathcal{M}CG(\Sigma,\partial \Sigma)$, such that
		the following diagram commutes:

		\begin{equation}
		\begin{tikzcd}
		\mathcal{M}T(\Sigma, \psi) \arrow[r, "\Psi" ] \arrow[d, "\pi"] & W \times
		S^1 \arrow[d]  \\
		S^1 \arrow[r, "Id"] & S^1 .
		\end{tikzcd}
		\end{equation}

		Considering $\partial D_i = S^1$, diagram $(2)$ implies that there is an embedding of $\mathcal{M}T(\Sigma, \tau_{\gamma_i})$ in $W \times \partial D_i,$ where $\tau_{\gamma_i}$ denotes the Dehn twist along the vanishing cycle $\gamma_i.$  Now take arcs connecting a point on $\partial D_i$ to a fixed regular value point $p \in D^2$ of the map $\pi$ as depicted in Figure~\ref{pic2}.

		\begin{figure}[htbp]
			\centering
			\def\svgwidth{10cm}
			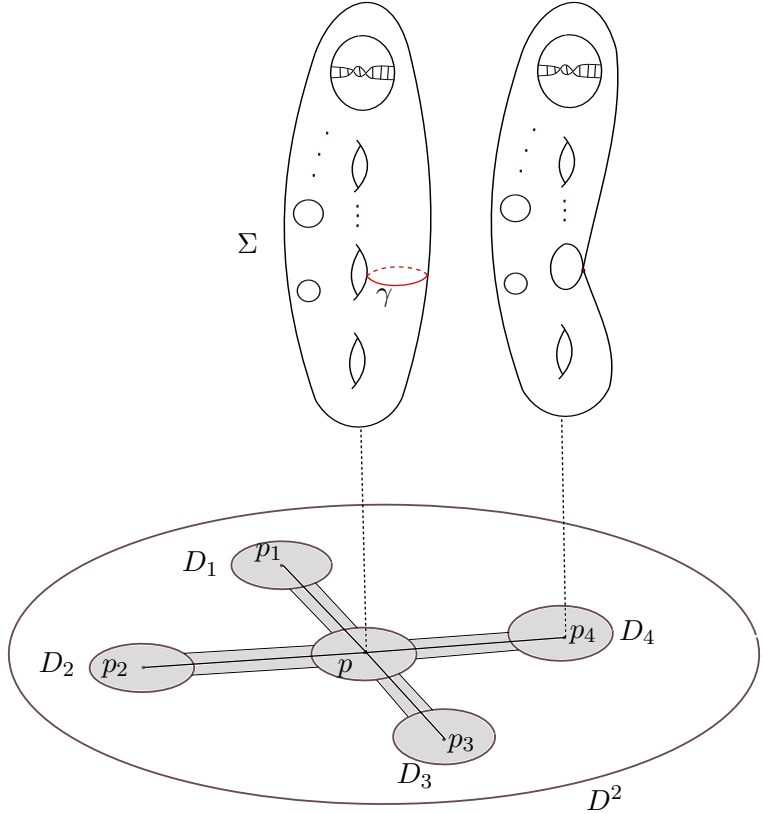
    		\caption{The figure depicts a particular case of a Lefschetz fibration $(V, \pi)$ over a disk with $4$ critical points, embedded as Lefschetz fibration  in the $W \times D^2$. The embedding is such that a generic fiber of $(V, \pi)$ has a flexible embedding in standard position in $W$. The curve $\gamma$ depicts a vanishing.}
			\label{pic2}
		\end{figure}
		
		\noindent Note that the Lefschetz fibration $(V, \pi)$ restricted to  a regular neighborhood $N$ of $D_i$'s together with arcs connecting them satisfies the following:
		
		\begin{enumerate}
			\item   $\pi^{-1} (\partial D_i)$ is the mapping torus 
			$\mathcal{M}T(\Sigma, \tau_{\gamma_i}).$
			
			\item  $V$ restricted to $\partial N$ is the mapping torus $\mathcal{M}T(\Sigma, \tau_{\gamma_1}\circ \tau_{\gamma_2}\circ \cdots \circ \tau_{\gamma_k}).$

		\end{enumerate}
		\noindent Thus, we get the required embedding $\widehat{f}$ such that diagram $(1)$ commutes.

		Next we show how to extend this embedding to produce a relative LF embedding $f$ of $(V,\partial V)$ in $W \times D^2.$ For this we need to use the fact that the embedding $\iota$ of $(\Sigma,\partial \Sigma)$ is in standard position. In particular, there exists an embedding $\iota_{\gamma_i}:\C^2 \hookrightarrow W$ which satisfies the second property listed in Definition~\ref{def:standard position}. Moreover, for each critical point $c_i$, the following commutative diagram holds: 
		
		\begin{equation}\label{diag:commutative_local_maps}
		\begin{tikzcd}
		U_i\subset V \arrow{r}{\phi_i}\arrow{d}{\pi}
		&\C^2\arrow{r}{i}\arrow{d}{g}
		&\C^2\times\C \arrow{r}{f_{c_i}}\arrow{d}{P}
		& W \times D^2 \arrow{d}{\pi_2}\\
		\widetilde{D}_i\arrow{r}{\phi}&\C\arrow{r}{Id}&\C \arrow{r}{\phi^{-1}}&\widetilde{D}_i ,\end{tikzcd}
		\end{equation}
		\noindent where the definitions of the maps appearing in the diagram are as follows:
		\begin{enumerate}
			\item $\phi_i:U_i\subset V \to \C^2$ and $\phi: \widetilde{D}_i\subset D^2 \to \C$ are orientation preserving parameterizations around critical point $c_i$ of $\pi$ and $\pi(c_i)$ respectively such that left square commutes in the diagram above,
			\item $i:\C ^2 \to \C^2 \times \C$ and $g:\C^2 \to \C$ are defined as
			$i(z_1,z_2)=(z_1,z_2,0)$ and $g(z_1,z_2)=z_1.z_2$,
			\item $f_{c_i}:\C^2\times\C \to  \times D^2$ and $P:\C^2\times\C \to \C$ are defined as
			
			\noindent $f_{c_i} (z_1,z_2,z_3)=(\iota_{\gamma_i}(z_1,z_2),\phi^{-1}(z_1.z_2+z_3))$ and $P(z_1,z_2,z_3)=z_1.z_2+z_3$.
		\end{enumerate}
		\noindent The commutativity of the middle square follows from definitions of  maps $g,i$ and $P$ and the commutativity of the last square follows from the definition of the map  $f_{c_i}$. Note that the commutative diagram~\ref{diag:commutative_local_maps} allows us to extend the embedding $\widehat{f}$ to the embedding $\widehat{f}_{c_i}$ of $ M_{c_i} = V \setminus (\sqcup_{i=1}^{k} \pi^{-1}(D_i) \cup U_i)$, because $\widehat{f}$ and $f_{c_i}\circ i\circ\phi_i$ agree on the overlapping region of the domain. Hence, $\widehat{f}$ and $f_{c_i}\circ i\circ\phi_i$ together defines  a map $\widehat{f}_{c_i}$. Moreover, $\widehat{f}_{c_i}$ satisfies the following commutative diagram :

		\begin{equation}
		\begin{tikzcd}
		M_{c_i} \arrow[r,"\widehat{f}_{c_i}"] \arrow[d, "\pi"] & \widehat{f}_{c_i}(M_{c_i}) \subset W \times D^2 \arrow[d, "\pi_2"] \\
		\pi(M_{c_i}) \subset D^2 \arrow[r, "Id"] & \pi_2(\widehat{f}_{c_i}) = \pi(M_{c_i}).   
		\end{tikzcd}
		\end{equation}

		Finally, we observe that by construction the embeddings $\widehat{f}_{c_i}$ and $\widehat{f}_{c_j}$ agree on $M_{c_i} \cap M_{c_j}.$ Since $V = \cup_{i=1}^k  M_{c_i}$, we get the required proper embedding $f$ of $V$ in $W \times D^2$.

	\end{proof}

\begin{proof}[Proof of Theorem \ref{main theorem}]
	Let $W = S^2 \times S^2 \setminus D^4$. Note that $W$ contains a separable Hopf link. The proof then follows from Lemma \ref{main lemma}.
\end{proof}

	\section{Every closed orientable $4$-manifold embeds in $S^2 \times S^2 \times S^2$} \label{sec3.5}
	
	The following result due to Etnyre--Fuller \cite{EF}(Proposition $12$, \cite{EF}) is our key ingredient.
	
	\begin{theorem}[Etnyre--Fuller]\label{EF}
	Let $X^4$ be a closed orientable $4$--manifold. Then we may write $X = Y_1 \cup Y_2,$ where $Y_i$ is a $2$--handlebody which admits an achiral Lefschetz fibration over $D^2$ with bounded fibers of the same genus, and with the induced open books on $\partial Y_1 = - \partial Y_2$ coinciding.
	\end{theorem}

    \begin{proof}[Proof of Corollary \ref{corollary}]
    	
    	Let $X^4$ be a closed oriented $4$--manifold. By Theorem \ref{EF}, there exist two achiral Lefschetz fibrations $Y_1$ and $Y_2$ such that $\partial Y_1 = - \partial Y_2$ and $X^4 = Y_1 \cup_\partial Y_2$. Moreover, the gluing map between $\partial Y_1$ and $\partial Y_2$ preserves the induced open book structure. By Theorem \ref{main theorem}, $(Y_i,\partial Y_i)$ admits relative LF embedding in $(S^2 \times S^2 \setminus D^4) \times D^2$ for $i = 1,2$. Note that the induced open book on $\partial Y_i$ admits an \emph{open book embedding} in $\mathcal{O}b((S^2 \times S^2 \setminus D^4),id) = \partial ((S^2 \times S^2 \setminus D^4) \times D^2).$ Since the embedding of $\Sigma$ in $(S^2 \times S^2 \setminus D^4)$ is flexible, the gluing map between $\partial Y_1$ and $\partial Y_2$ can be induced by an isotopy of $\mathcal{O}b(S^2 \times S^2 \setminus D^4, id) = \partial ((S^2 \times S^2 \setminus D^4) \times D^2).$ To see this, let $u$ be a homeomorphism between $\partial Y_1$ and $\partial Y_2$ such that $u$ restricts to a homeomorphism from a page of $\partial Y_1$ to a page of $\partial Y_2.$ Let $\Phi_t$ be a relative isotopy of $((S^2 \times S^2 \setminus D^4) \times D^2,\partial ((S^2 \times S^2 \setminus D^4) \times D^2))$ such that $\Phi_0 = id$ and $\Phi_1$ induces the map $u$ on the embedded page of $\partial Y_1$. We can extend $\Phi_t$ to an isotopy $\hat{\Phi}_t$ of $\mathcal{O}b(S^2 \times S^2 \setminus D^4,id).$ Note that $\hat{\Phi}_1|_{\partial Y_1} = u.$ Therefore, $X^4 = Y_1 \cup_\partial Y_2$ admits embedding in $(S^2 \times S^2 \setminus D^4) \times D^2 \cup_{\partial,id} (S^2 \times S^2 \setminus D^4) \times D^2 \subset S^2 \times S^2 \times D^2 \cup_{\partial,id} S^2 \times S^2 \times D^2 = S^2 \times S^2 \times S^2$.
    	
    \end{proof}

	\section{Relative LF embedding of $LF(\Sigma_{g,1},\phi)$ in $D^6$} \label{sec4}

	\begin{figure}[htbp]
		\centering
		\def\svgwidth{8cm}
		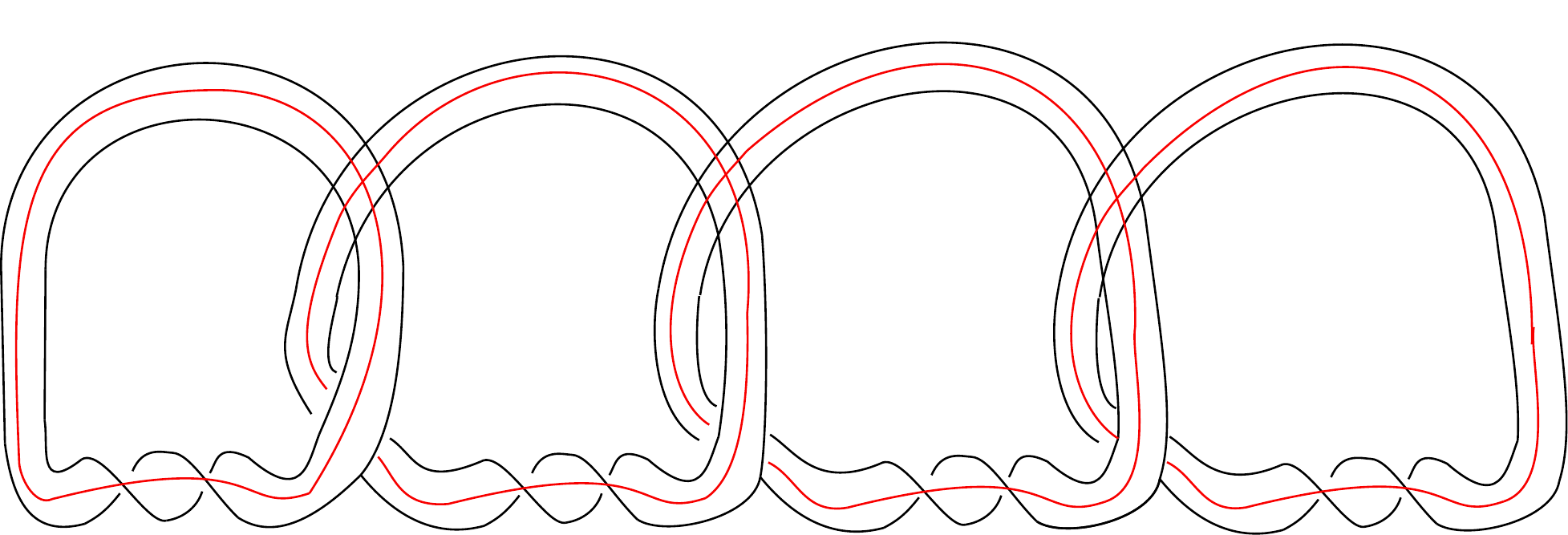
		\caption{The embedding $f_0$ for $g = 2.$}
		\label{picl}
	\end{figure}

	\begin{proof}[Proof of Theorem \ref{theorem 2}]
		
		Recall the Whitney sum formula for direct sum of vector bundles : $w_k(E_1 \oplus E_2) = \Sigma_{i+j=k}w_i(E_1) \cup w_j(E_2)$.
		
			\begin{enumerate}
			\item If $V^4$ embeds in $D^6$ via $h$, then the normal bundle of $h$ is a trivial disk bundle over $V$. Therefore, $TV \oplus \epsilon^2_V = h^*(TD^6)$. Thus, $w_2(V) = 0$.
			
			\item  Note that $\Sigma_{g,1}$ is homeomorphic to plumbing of $2g$ Hopf annulii in a chain, as in Figure \ref{picl}. Let us denote this embedding of $\Sigma_{g,1}$ in $S^3$ by $f_0$. We can modify this embedding using the notations used in the proof of Lemma \ref{main lemma}. We embed $\Sigma_{g,1}$ by $f_0$ in $S^3 \times \frac{3}{2} \subset S^3 \times [1,2] \subset D^4$ and attach the cylinder $\partial \Sigma_{g,1} \times [\frac{3}{2}, 2]$ to get a proper embedding $\tilde{f}_0$ of $(\Sigma_{g,1}, \partial \Sigma_{g,1})$ in $(D^4, \partial D^4)$. We claim that $\tilde{f}_0$ is flexible and is in standard position.

		The claim is true for the following reason. In the proof of Lemma \ref{rel GP1}, the existence of separable Hopf link was assumed so that every curve that is either a Lickorish generator or does not intersect a handle, should have a neighborhood isotopic to the Hopf annulus in $S^3$. Here, we already have a monodromy that factors into Dehn twists along the curves : $\{b_1, a_1, c_1,\cdots a_g, b_g \}$, and each of these curves has a Hopf annulus neighborhood under $\tilde{f}_0$. Thus, $\tilde{f}_0$ is both flexible and in standard position. One can then procced exactly as in the proof of Lemma \ref{main lemma} to conclude.   
			
		\end{enumerate}
	
	\end{proof}

\subsection{Obstruction to LF embedding in $D^6$}The criteria of Stipsciz for non-spin ALF with closed fibers gives obstructions to existence of relative LF embeddings in $D^6$. For example, let $V^4 = LF(\Sigma_{g,1}, \phi)$ be an achiral Lefschetz fibration that has a separable curve as vanishing cycle. The double of $\Sigma_{g,1}$ is a closed surface of genus $2g$, say $\Sigma_{2g}$. Let $\tilde{\phi}$ be the extended monodromy on $\Sigma_{2g}$ (by taking union). Then by Theorem \ref{stipsicz}, $LF(\Sigma_{2g}, \tilde{\phi})$ is not spin. Thus, by statement $(1)$ of Theorem \ref{theorem 2}, $LF(\Sigma_{2g}, \tilde{\phi})$ cannot embed in $\mathbb{R}^6$. But if $LF(\Sigma_{g,1}, \phi)$ properly embeds in $D^6$, then $LF(\Sigma_{2g}, \tilde{\phi})$ embeds in $S^6$. Which implies that $LF(\Sigma_{2g}, \tilde{\phi})$ embeds in $\mathbb{R}^6$, which is a contradiction. Therefore, $LF(\Sigma_{g,1}, \phi)$ does not admit a relative LF embedding in $D^6$. 

Similar examples can be found with achiral Lefschetz fibrations having non-separable vanishing cycles. In particular, refer to Figure \ref{intropic} and consider $ V_0^4 = LF(\Sigma_{g,1}, \tau_{b_1}\circ \tau_{c_1} \circ \tau_{b_2})$. Then taking $k =2$, $v = b_2$, $v_1 = b_1$ and $v_2 = c_1$ as in Theorem \ref{stipsicz}, we see that $V_0^4$ does not admit a relative LF embedding in $D^6$.

\bibliographystyle{amsplain}

\end{document}

%% file: pic0.pdf_tex
\begingroup%
  \makeatletter%
  \providecommand\color[2][]{%
    \errmessage{(Inkscape) Color is used for the text in Inkscape, but the package 'color.sty' is not loaded}%
    \renewcommand\color[2][]{}%
  }%
  \providecommand\transparent[1]{%
    \errmessage{(Inkscape) Transparency is used (non-zero) for the text in Inkscape, but the package 'transparent.sty' is not loaded}%
    \renewcommand\transparent[1]{}%
  }%
  \providecommand\rotatebox[2]{#2}%
  \newcommand*\fsize{\dimexpr\f@size pt\relax}%
  \newcommand*\lineheight[1]{\fontsize{\fsize}{#1\fsize}\selectfont}%
  \ifx\svgwidth\undefined%
    \setlength{\unitlength}{438.11565511bp}%
    \ifx\svgscale\undefined%
      \relax%
    \else%
      \setlength{\unitlength}{\unitlength * \real{\svgscale}}%
    \fi%
  \else%
    \setlength{\unitlength}{\svgwidth}%
  \fi%
  \global\let\svgwidth\undefined%
  \global\let\svgscale\undefined%
  \makeatother%
  \begin{picture}(1,0.89846577)%
    \lineheight{1}%
    \setlength\tabcolsep{0pt}%
    \put(0,0){\includegraphics[width=\unitlength,page=1]{pic0.pdf}}%
    \put(0.33208221,0.63333242){\color[rgb]{0,0,0}\makebox(0,0)[lt]{\lineheight{1.25}\smash{\begin{tabular}[t]{l}$\alpha_1$\end{tabular}}}}%
    \put(0.18654281,0.26710639){\color[rgb]{0,0,0}\makebox(0,0)[lt]{\lineheight{1.25}\smash{\begin{tabular}[t]{l}$\partial_1$\end{tabular}}}}%
    \put(0.48321336,0.8555666){\color[rgb]{0,0,0}\makebox(0,0)[lt]{\lineheight{1.25}\smash{\begin{tabular}[t]{l}$\beta_1$\end{tabular}}}}%
    \put(0.48053637,0.45117759){\color[rgb]{0,0,0}\makebox(0,0)[lt]{\lineheight{1.25}\smash{\begin{tabular}[t]{l}$\gamma_1$\end{tabular}}}}%
    \put(0.39235674,0.26777468){\color[rgb]{0,0,0}\makebox(0,0)[lt]{\lineheight{1.25}\smash{\begin{tabular}[t]{l}$\partial_2$\end{tabular}}}}%
    \put(0.76388029,0.2797535){\color[rgb]{0,0,0}\makebox(0,0)[lt]{\lineheight{1.25}\smash{\begin{tabular}[t]{l}$\partial_m$\end{tabular}}}}%
  \end{picture}%
\endgroup%

%% file: pic1.pdf_tex
\begingroup%
  \makeatletter%
  \providecommand\color[2][]{%
    \errmessage{(Inkscape) Color is used for the text in Inkscape, but the package 'color.sty' is not loaded}%
    \renewcommand\color[2][]{}%
  }%
  \providecommand\transparent[1]{%
    \errmessage{(Inkscape) Transparency is used (non-zero) for the text in Inkscape, but the package 'transparent.sty' is not loaded}%
    \renewcommand\transparent[1]{}%
  }%
  \providecommand\rotatebox[2]{#2}%
  \newcommand*\fsize{\dimexpr\f@size pt\relax}%
  \newcommand*\lineheight[1]{\fontsize{\fsize}{#1\fsize}\selectfont}%
  \ifx\svgwidth\undefined%
    \setlength{\unitlength}{516.58162479bp}%
    \ifx\svgscale\undefined%
      \relax%
    \else%
      \setlength{\unitlength}{\unitlength * \real{\svgscale}}%
    \fi%
  \else%
    \setlength{\unitlength}{\svgwidth}%
  \fi%
  \global\let\svgwidth\undefined%
  \global\let\svgscale\undefined%
  \makeatother%
  \begin{picture}(1,0.35358119)%
    \lineheight{1}%
    \setlength\tabcolsep{0pt}%
    \put(0,0){\includegraphics[width=\unitlength,page=1]{pic1.pdf}}%
    \put(0.61415285,0.03309739){\color[rgb]{0,0,0}\makebox(0,0)[lt]{\lineheight{1.25}\smash{\begin{tabular}[t]{l}$l_1$\end{tabular}}}}%
    \put(0.68642966,0.00239812){\color[rgb]{0,0,0}\makebox(0,0)[lt]{\lineheight{1.25}\smash{\begin{tabular}[t]{l}$l_2$\end{tabular}}}}%
    \put(0.03285768,0.20140797){\color[rgb]{0,0,0}\makebox(0,0)[lt]{\lineheight{1.25}\smash{\begin{tabular}[t]{l}$\Sigma$\end{tabular}}}}%
    \put(0,0){\includegraphics[width=\unitlength,page=2]{pic1.pdf}}%
    \put(0.22554594,0.1757407){\color[rgb]{0,0,0}\makebox(0,0)[lt]{\lineheight{1.25}\smash{\begin{tabular}[t]{l}$C$\end{tabular}}}}%
    \put(0.3991767,0.1548155){\color[rgb]{0,0,0}\makebox(0,0)[lt]{\lineheight{1.25}\smash{\begin{tabular}[t]{l}$C_H$\end{tabular}}}}%
    \put(0,0){\includegraphics[width=\unitlength,page=3]{pic1.pdf}}%
    \put(0.83648749,0.23083782){\color[rgb]{0,0,0}\makebox(0,0)[lt]{\lineheight{1.25}\smash{\begin{tabular}[t]{l}$S^3 \times \{\frac{3}{2}\}$\end{tabular}}}}%
    \put(0.83792829,0.0707431){\color[rgb]{0,0,0}\makebox(0,0)[lt]{\lineheight{1.25}\smash{\begin{tabular}[t]{l}$S^3 \times \{2\}$\end{tabular}}}}%
  \end{picture}%
\endgroup%

%% file: pic2.pdf_tex
\begingroup%
  \makeatletter%
  \providecommand\color[2][]{%
    \errmessage{(Inkscape) Color is used for the text in Inkscape, but the package 'color.sty' is not loaded}%
    \renewcommand\color[2][]{}%
  }%
  \providecommand\transparent[1]{%
    \errmessage{(Inkscape) Transparency is used (non-zero) for the text in Inkscape, but the package 'transparent.sty' is not loaded}%
    \renewcommand\transparent[1]{}%
  }%
  \providecommand\rotatebox[2]{#2}%
  \newcommand*\fsize{\dimexpr\f@size pt\relax}%
  \newcommand*\lineheight[1]{\fontsize{\fsize}{#1\fsize}\selectfont}%
  \ifx\svgwidth\undefined%
    \setlength{\unitlength}{419.2531702bp}%
    \ifx\svgscale\undefined%
      \relax%
    \else%
      \setlength{\unitlength}{\unitlength * \real{\svgscale}}%
    \fi%
  \else%
    \setlength{\unitlength}{\svgwidth}%
  \fi%
  \global\let\svgwidth\undefined%
  \global\let\svgscale\undefined%
  \makeatother%
  \begin{picture}(1,1.07770845)%
    \lineheight{1}%
    \setlength\tabcolsep{0pt}%
    \put(0,0){\includegraphics[width=\unitlength,page=1]{pic2.pdf}}%
    \put(0.4387844,0.185872){\color[rgb]{0,0,0}\makebox(0,0)[lt]{\lineheight{1.25}\smash{\begin{tabular}[t]{l}$p$\end{tabular}}}}%
    \put(0.04105041,0.18625867){\color[rgb]{0,0,0}\makebox(0,0)[lt]{\lineheight{1.25}\smash{\begin{tabular}[t]{l}$D_2$\end{tabular}}}}%
    \put(0.74685011,0.23190921){\color[rgb]{0,0,0}\makebox(0,0)[lt]{\lineheight{1.25}\smash{\begin{tabular}[t]{l}$p_4$\end{tabular}}}}%
    \put(0.5171052,0.03780214){\color[rgb]{0,0,0}\makebox(0,0)[lt]{\lineheight{1.25}\smash{\begin{tabular}[t]{l}$D_3$\end{tabular}}}}%
    \put(0.81272389,0.22891536){\color[rgb]{0,0,0}\makebox(0,0)[lt]{\lineheight{1.25}\smash{\begin{tabular}[t]{l}$D_4$\end{tabular}}}}%
    \put(0.23168503,0.32010856){\color[rgb]{0,0,0}\makebox(0,0)[lt]{\lineheight{1.25}\smash{\begin{tabular}[t]{l}$D_1$\end{tabular}}}}%
    \put(0.58547062,0.08681753){\color[rgb]{0,0,0}\makebox(0,0)[lt]{\lineheight{1.25}\smash{\begin{tabular}[t]{l}$p_3$\end{tabular}}}}%
    \put(0.12510367,0.18613823){\color[rgb]{0,0,0}\makebox(0,0)[lt]{\lineheight{1.25}\smash{\begin{tabular}[t]{l}$p_2$\end{tabular}}}}%
    \put(0.32913117,0.34075094){\color[rgb]{0,0,0}\makebox(0,0)[lt]{\lineheight{1.25}\smash{\begin{tabular}[t]{l}$p_1$\end{tabular}}}}%
    \put(0.76867575,0.00294922){\color[rgb]{0,0,0}\makebox(0,0)[lt]{\lineheight{1.25}\smash{\begin{tabular}[t]{l}$D^2$\end{tabular}}}}%
    \put(0.30482228,0.74303013){\color[rgb]{0,0,0}\makebox(0,0)[lt]{\lineheight{1.25}\smash{\begin{tabular}[t]{l}$\Sigma$\end{tabular}}}}%
    \put(0.4897991,0.68036239){\color[rgb]{0,0,0}\makebox(0,0)[lt]{\lineheight{1.25}\smash{\begin{tabular}[t]{l}$\gamma$\end{tabular}}}}%
  \end{picture}%
\endgroup%

%% file: hopfplumb.pdf_tex
\begingroup%
  \makeatletter%
  \providecommand\color[2][]{%
    \errmessage{(Inkscape) Color is used for the text in Inkscape, but the package 'color.sty' is not loaded}%
    \renewcommand\color[2][]{}%
  }%
  \providecommand\transparent[1]{%
    \errmessage{(Inkscape) Transparency is used (non-zero) for the text in Inkscape, but the package 'transparent.sty' is not loaded}%
    \renewcommand\transparent[1]{}%
  }%
  \providecommand\rotatebox[2]{#2}%
  \newcommand*\fsize{\dimexpr\f@size pt\relax}%
  \newcommand*\lineheight[1]{\fontsize{\fsize}{#1\fsize}\selectfont}%
  \ifx\svgwidth\undefined%
    \setlength{\unitlength}{562.90097762bp}%
    \ifx\svgscale\undefined%
      \relax%
    \else%
      \setlength{\unitlength}{\unitlength * \real{\svgscale}}%
    \fi%
  \else%
    \setlength{\unitlength}{\svgwidth}%
  \fi%
  \global\let\svgwidth\undefined%
  \global\let\svgscale\undefined%
  \makeatother%
  \begin{picture}(1,0.34106222)%
    \lineheight{1}%
    \setlength\tabcolsep{0pt}%
    \put(0,0){\includegraphics[width=\unitlength,page=1]{hopfplumb.pdf}}%
    \put(0.08993649,0.31760314){\color[rgb]{0,0,0}\makebox(0,0)[lt]{\lineheight{1.25}\smash{\begin{tabular}[t]{l}$b_1$\end{tabular}}}}%
    \put(0.32030187,0.31511272){\color[rgb]{0,0,0}\makebox(0,0)[lt]{\lineheight{1.25}\smash{\begin{tabular}[t]{l}$a_1$\end{tabular}}}}%
    \put(0.57059074,0.32756489){\color[rgb]{0,0,0}\makebox(0,0)[lt]{\lineheight{1.25}\smash{\begin{tabular}[t]{l}$c_1$\end{tabular}}}}%
    \put(0.82586049,0.32507444){\color[rgb]{0,0,0}\makebox(0,0)[lt]{\lineheight{1.25}\smash{\begin{tabular}[t]{l}$a_2$\end{tabular}}}}%
  \end{picture}%
\endgroup%

%% file: Lefschetz_embeddings_and_its_obstructions.bbl
\begin{thebibliography}{14}
	

	
	\bibitem{EF}John B. Etnyre, Terry Fuller, Realizing 4-manifolds as achiral Lefschetz fibrations, {\it International Mathematics Research Notices}, Volume 2006, 2006.
	
	\bibitem{EL} J. Etnyre  and Y. Lekili, Embedding all contact $3$--manifolds in a fixed contact $5$--manifold, {\it Journal of the London Mathematical Society}, Volume 99, Issue 1, February 2019, 52-68.	
	
	\bibitem{Et}  J. Etnyre, Lectures on open book decompositions and contact structures, 
	{\it Floer homology, gauge theory, and low-dimensional topology} 5,  103--141.
	
	\bibitem{FM} B. Farb and D. Margalit, A primer on mapping class groups, Version 5.0, { \it Princeton University Press}.

    \bibitem{GNS} A. Ghanwat, A. Nath, K. Saha, Lefschetz fibration embedding of $4$--manifolds in $\mathbb{D}^5 \times \mathbb{D}^2$, {\it preprint}.
	
	\bibitem{GP} A. Ghanwat, D. M. Pancholi, Embeddings of $4$-manifolds in $\mathbb{C}P^3$,  arXiv:2002.11299v2 (2020). 
	
	
    
    \bibitem{Li} W.B.R  Lickorish.  \textit{A representation of orientable combinatorial $3$--manifolds}.  Ann. of Math., Vol. 76(2), (1962),  531–540.
   
		
	\bibitem{PPS} D. M. Pancholi, S. Pandit, K. Saha, Embedding of $3$-manifolds via open books, arXiv:1806.09784v2, 2018.
	
	\bibitem{S} Saha, K. On Open Book Embedding of Contact Manifolds in the Standard Contact Sphere. {\it Canadian Math. Bulletin}, 1-16 (2019), doi:10.4153/S0008439519000808.
	
    \bibitem{St} A. Stipsicz, Spin structures on Lefschetz fibrations, {Bull. London Math. Soc.} 33 (2001) 466-472.

	\bibitem{Wu} W. T. Wu, On the isotopy of $C^r$-manifolds of dimension $n$ in Euclidean $(2n+1)$-space, Sci. Record (N.S.) 2 271-275 (1958).

	
\end{thebibliography}
